\documentclass[11pt]{article}

\usepackage{amsmath,amsfonts,amsthm,amssymb,verbatim,graphicx,subfigure,color,fullpage}
\usepackage{mathtools}

\newtheorem{lemma}{Lemma}
\newtheorem{theorem}{Theorem}
\newtheorem*{theorem*}{Theorem}
\newtheorem{proposition}{Proposition}

\newtheorem{conjecture}{Conjecture}
\newtheorem{remark}{Remark}

\title{Shifts of the Stable Kneser Graphs and Hom-Idempotence\footnote{Partially supported by MathAmSud Project 13MATH-07
(Argentina--Brazil--Chile--France) and by the French ANR JCJC CombPhysMat2Tens grant.}}

\author{Pablo Torres \footnote{Universidad Nacional de Rosario and CONICET, Rosario, Argentina. e-mail: ptorres@fceia.unr.edu.ar}
  \and Mario Valencia-Pabon\footnote{Universit\'e Paris-13, Sorbonne Paris Cit\'e LIPN, CNRS
UMR7030, Villetaneuse, France. e-mail: valencia@lipn.univ-paris13.fr}}
\date{}

\begin{document}
\maketitle
\sloppy

\begin{abstract}
A graph $G$ is said to be {\em hom-idempotent} if there is a homomorphism from $G^2$ to $G$, and {\em weakly hom-idempotent} if for some $n \geq 1$ there is a homomorphism from $G^{n+1}$ to $G^n$. Larose et al. [{\em Eur. J. Comb. 19:867-881, 1998}] proved that Kneser graphs $\operatorname{KG}(n,k)$ are not weakly hom-idempotent for $n \geq 2k+1$, $k\geq 2$. For $s \geq 2$, we characterize all the shifts (i.e., automorphisms of the graph that map every vertex to one of its neighbors) of $s$-stable Kneser graphs $\operatorname{KG}(n,k)_{s-\operatorname{stab}}$ and we show that $2$-stable Kneser graphs are not weakly hom-idempotent, for $n \geq 2k+2$, $k \geq 2$. Moreover, for $s,k\geq 2$, we prove that $s$-stable Kneser graphs $\operatorname{KG}(ks+1,k)_{s-\operatorname{stab}}$ are circulant graphs and so hom-idempotent graphs. Finally, for $s \geq 3$, we show that $s$-stable Kneser graphs $\operatorname{KG}(2s+2,2)_{s-\operatorname{stab}}$ are cores, not $\chi$-critical, not hom-idempotent and their chromatic number is equal to $s+2$.\\

\noindent {\bf Keywords}: Cartesian product of graphs, Stable Kneser graphs, Cayley graphs, Hom-idempotent graphs.
\end{abstract}


\section{Introduction}\label{sec:intro}
Let $[n]$ denote the set $\{1,\ldots,n\}$. For positive integers $n \geq 2k$, the Kneser graph $\operatorname{KG}(n,k)$ has as vertices the $k$-subsets of $[n]$ and two vertices are connected by an edge if they have empty intersection.  In a famous paper, Lov\'asz \cite{Lov78} showed that its chromatic number $\chi(K(n,k))$ is equal to $n-2k+2$. After this result, Schrijver \cite{Sch78} proved that the chromatic number remains the same when we consider the subgraph $\operatorname{KG}(n,k)_{2-\operatorname{stab}}$ of $\operatorname{KG}(n,k)$ obtained by restricting the vertex set to the $k$-subsets that are {\em $2$-stable}, that is, that do not contain two consecutive elements of $[n]$ (where $1$ and $n$ are considered also to be consecutive). Schrijver \cite{Sch78} also proved that the $2$-stable Kneser graphs are {\em vertex critical} (or {\em $\chi$-critical}), i.e. the chromatic number of any proper subgraph of $\operatorname{KG}(n,k)_{2-\operatorname{stab}}$ is strictly less than $n-2k+2$; for this reason, the $2$-stable Kneser graphs are also 
 known as the Schrijver graphs. After these general advances, a lot of work has been done concerning properties of Kneser graphs and stable Kneser graphs (see \cite{Frankl85,FF86,LLT98,Braun10,Meun11} and references therein). For example, it is well known that for $n \geq 2k+1$ the automorphism group of the Kneser graph $\operatorname{KG}(n,k)$ is the symmetric group induced by the permutation action on $[n]$; see \cite{GodRoy01} for a textbook account. Concerning the automorphism group of the $s$-stable Kneser graphs $\operatorname{KG}(n,k)_{s-\operatorname{stab}}$, Braun \cite{Braun10} proved that, for $s = 2$, it is isomorphic to the dihedral group of order $2n$. Recently, Torres \cite{Torres15} has generalized Braun's result by proving that for any $s \geq 2$, $\operatorname{Aut}(\operatorname{KG}(n,k)_{s-\operatorname{stab}})$ is indeed isomorphic to the dihedral group of order $2n$.

The {\em cartesian product} $G \Box H$ of two graphs $G$ and $H$ has vertex set $V(G) \times V(H)$, two vertices being joined by an edge whenever they have one coordinate equal and the other adjacent. This product is commutative and associative up to isomorphism.  

In this paper, we assume that the graphs are finite. An {\em homomorphism} from a graph $G$ into a graph $H$, denoted by $G \to H$, is an edge-preserving map from $V(G)$ to $V(H)$. If $H$ is a subgraph of $G$ and $\phi: G \to H$ has the property that $\phi(u) = u$ for every vertex $u$ of $H$, then $\phi$ is called a {\em retraction} and $H$ is called a {\em retract} of $G$. If $\phi: G \to H$ is a bijection and $\phi^{-1}$ is also a homomorphism from $H$ to $G$, then $\phi$ is an {\em isomorphism} and we write $G \simeq H$. In particular, if $G$ is finite, a bijective homomorphism from $G$ to himself is an {\em automorphism}. Two graphs $G$ and $H$ are {\em homomorphically equivalent}, denoted by $G \leftrightarrow H$, if $G \to H$ and $H \to G$. A graph $G$ is called a {\em core} if it has no proper retracts, i.e., any homomorphism $\phi: G \to G$ is an automorphism of $G$. It is well known that any finite graph $G$ is homomorphically equivalent to at least one core $G^{\bullet}$, as can be seen by selecting $G^{\bullet}$ as a retract 
 of $G$ with a minimum number of vertices. In this way, $G^{\bullet}$ is uniquely determined up to isomorphism, and it makes sense to think of it as {\em the core} of $G$. It is widely known that Kneser graphs are cores. Moreover, it is not difficult to deduce that any $\chi$-critical graph is a core. Therefore, any $2$-stable Kneser graph is also a core, because it is $\chi$-critical \cite{Sch78}.

An {\em automorphism} $\phi$ of a graph $G$ is called a {\em shift} of $G$ if $\{u,\phi(u)\} \in E(G)$ for each $u \in V(G)$. In other words, a shift of $G$ maps every vertex to one of its neighbors \cite{LLT98}.

Let $A$ be a group and $S$ a subset of $A$ that is closed under inverses and does not contain the identity. The {\em Cayley gragh} $\operatorname{Cay}(A,S)$ is the graph whose vertex set is $A$, two vertices $u,v$ being joined by an edge if $u^{-1}v \in S$. Cayley graphs of cyclic groups are often called {\em circulants}.

A graph $G$ is said {\em vertex-transitive} if its automorphism group $\operatorname{Aut}(G)$ acts transitively on its vertex-set. It's well known that Cayley graphs and Kneser graphs are vertex-transitive. However, $2$-stable Kneser graphs are not vertex-transitive in general. For example, no automorphism of $\operatorname{KG}(6,2)_{2-\operatorname{stab}}$ sends $\{1,3\}$ to $\{1,4\}$, since $\operatorname{Aut}(\operatorname{KG}(6,2)_{2-\operatorname{stab}})$ is isomorphic to the dihedral group of order $12$ acting on the set $\{1,2,\ldots,6\}$.

We write $G^n$ for the $n$-fold cartesian product of a graph $G$. A graph $G$ is said {\em hom-idempotent} if there is a homomorphism from $G^2$ to $G$, and {\em weakly hom-idempotent} if for some $n \geq 1$ there is a homomorphism from $G^{n+1}$ to $G^n$. Larose et al. \cite{LLT98} showed that the Kneser graphs are not weakly hom-idempotent.  However, the technique used by Larose et al. \cite{LLT98} cannot be extended directly to the $s$-stable Kneser graphs.

A subset $S \subseteq [n]$ is {\em $s$-stable} if any two of its elements are at least "at distance $s$ apart" on the $n$-cycle, that is, if $s \leq |i-j| \leq n-s$ for distinct $i,j \in S$. For $s,k \geq 2$ and $n \geq ks$, the $s$-stable Kneser graph $\operatorname{KG}(n,k)_{s-\operatorname{stab}}$ is the subgraph of $\operatorname{KG}(n,k)$ obtained by restricting the vertex set of $\operatorname{KG}(n,k)$ to the $s$-stable $k$-subsets of $[n]$.

In this paper, we characterize all the shifts of $s$-stable Kneser graphs. As a by-product we show that almost all Schrijver graphs are not weakly hom-idempotent. Moreover, for $s,k\geq 2$, we show that $s$-stable Kneser graphs $\operatorname{KG}(ks+1,k)_{s-\operatorname{stab}}$ are circulant graphs and so hom-idempotent graphs. Finally, we study some properties of the $s$-stable Kneser graph $\operatorname{KG}(2s+2,2)_{s-\operatorname{stab}}$ for $s \geq 3$. We prove that for all $s \geq 3$, the graphs $\operatorname{KG}(2s+2,2)_{s-\operatorname{stab}}$ are cores, not $\chi$-critical and not hom-idempotent. Moreover, we also prove that Meunier's conjecture \cite{Meun11} concerning the chromatic number of $s$-stables Kneser graphs holds for this family of graphs, that is, we prove that $\chi(\operatorname{KG}(2s+2,2)_{s-\operatorname{stab}}) = s+2$. We end our paper with a conjecture  concerning the not hom-idempotence of $s$-stable Kneser graphs.

In the sequel, we will use the term {\em modulo $[n]$} to denote arithmetic operations on the set $[n]$ where $n$ represents the $0$.
\section{Shifts of $s$-stable Kneser graphs}
As we have mentioned in the previous section, Braun \cite{Braun10} and Torres \cite{Torres15} showed that the automorphism group of the $s$-stable Kneser graph $\operatorname{KG}(n,k)_{s-\operatorname{stab}}$ is isomorphic to the dihedral group $D_{2n}$ of order $2n$, where the group isomorphism $\phi:D_{2n}\mapsto\operatorname{Aut}(\operatorname{KG}(n,k)_{s-\operatorname{stab}})$ is such that $\phi(\alpha)(\{i_1,$ $i_2,\ldots,i_k\})=\{\alpha(i_1),\alpha(i_2),\dots,\alpha(i_k)\}$. For convenience, from now on we do not distinguish between elements of $\operatorname{Aut}(\operatorname{KG}(n,k)_{s-\operatorname{stab}})$ and elements of $D_{2n}$.
We denote the elements of $D_{2n}$ as follows (arithmetic operations are taken modulo $[n]$):
\begin{itemize}
\item {\em Rotations}: Let $\sigma^0$ be the identity permutation on $[n]$ and, for $1 \leq i \leq n-1$, let $\sigma^i = \sigma^{i-1}\circ\sigma^1$, where $\sigma^1$ is the circular permutation $(1,2,\ldots,n-1,n)$.
\item {\em Reflexions}:
\begin{itemize}
\item Case $n$ odd. For $1\leq i \leq n$, let $\rho_i $ be the permutation formed by the product of the transpositions $ (i+1,i-1)(i+2,i-2)\ldots(i + \frac{n-1}{2}, i - \frac{n-1}{2})$, where $i$ is a fix point.
\item Case $n$ even. For $1 \leq i \leq \frac{n}{2}$, we have two types of reflexions: let $\rho_i$ be the permutation formed by the product of the transpositions $ (i+1,i-1)(i+2,i-2)\ldots(i + \frac{n}{2} - 1, i - \frac{n}{2}+1)$, where $i$ and $i + \frac{n}{2}$ are fix points; and let $\delta_i$ be the permutation formed by the product of transpositions $(i,i-1)(i+1,i-2)\ldots(i+\frac{n}{2}-1,i-\frac{n}{2})$ without fix point.
\end{itemize}
\end{itemize}

In the following lemmas, we will to characterize all the shifts of stable Kneser graphs.

\begin{lemma}\label{reflexions}
Let $n\geq ks+1$. Then, the reflexions are not shifts of the $s$-stable Kneser graph $\operatorname{KG}(n,k)_{s-\operatorname{stab}}$.
\end{lemma}
\begin{proof}
Let us consider the following two cases:
\begin{itemize}
\item Case $n$ odd. For each $1 \leq i \leq n$, let $v_i$ be a vertex in $\operatorname{KG}(n,k)_{s-\operatorname{stab}}$ such that $i \in v_i$. Trivially, such vertex $v_i$ always exists. Now, we know that $i$ is a fix point under the permutation $\rho_i$ and thus, $i \in \rho_i(v_i)$ which implies that $\{v_i,\rho_i(v_i)\}$ is not an edge of $\operatorname{KG}(n,k)_{s-\operatorname{stab}}$. Thus, for $1 \leq i \leq n$, $\rho_i$ is not a shift of $\operatorname{KG}(n,k)_{s-\operatorname{stab}}$.
\item Case $n$ even. Analogous to the previous case, we can show that $\rho_i$ is not a shift of $\operatorname{KG}(n,k)_{s-\operatorname{stab}}$, for $1 \leq i \leq \frac{n}{2}$. Now, for each $1 \leq i \leq \frac{n}{2}$, let $v_i = \{i,i+s,i+2s,\ldots,i+(k-2)s,i-s-1\}$. Clearly, $v_i$ is an $s$-stable set, since $i+(k-2)s$ and $i-s-1$ are at least at distance $s$ apart on the $n$-cycle. So, $v_i$ is a vertex of $\operatorname{KG}(n,k)_{s-\operatorname{stab}}$ such that $\{i+s,i-s-1\} \subseteq v_i$. However, $\{i+s,i-s-1\} \subseteq \delta_i(v_i)$ which implies that $\{v_i,\delta_i(v_i)\}$ is not an edge of $\operatorname{KG}(n,k)_{s-\operatorname{stab}}$. Thus, for $1 \leq i \leq \frac{n}{2}$, $\delta_i$ is not a shift of $\operatorname{KG}(n,k)_{s-\operatorname{stab}}$.
\end{itemize}
\end{proof}

\begin{lemma}
\label{lema-shift}
Let $n \geq (k+1)s-1$. Then, the only $2(s-1)$ shifts of the $s$-stable Kneser graph $\operatorname{KG}(n,k)_{s-\operatorname{stab}}$ are the rotations $\sigma^i$ with $i\in\{1,\dots,s-1\}\cup\{n-s+1,\dots,n-1\}$.
\end{lemma}
\begin{proof}
From Lemma \ref{reflexions} we only need to study the rotations $\sigma^i$ for $i\in [n-1]$. It is very easy to deduce that the circular permutations $\sigma^i$ with $i\in\{1,\dots,s-1\}\cup\{n-s+1,\dots,n-1\}$ are shifts of the graph $\operatorname{KG}(n,k)_{s-\operatorname{stab}}$. In order to prove that they are the only $2(s-1)$ shifts of $\operatorname{KG}(n,k)_{s-\operatorname{stab}}$, we will proceed by cases. The arithmetic operations are taken modulo $[n]$.
Clearly, the identity permutation $\sigma^0$ is not a shift. Now, we claim that for each $i \in \{s,s+1,\ldots,n-s\}$, there exists a vertex $v_i$ in $\operatorname{KG}(n,k)_{s-\operatorname{stab}}$ such that $\{1,i+1\} \subseteq v_i$. In fact, vertex $v_i$ can be computed as follows:

\begin{itemize}
\item If $s \leq i \leq ks-1$, let $j$ such that $js\leq i\leq (j+1)s-1$ and $v_i=\{1+ts:t=0,\dots,j-1\}\cup\{1+i+ts:t=0,\dots,k-j-1\}$.

\item If $ks \leq i \leq n-s$ then, set $v_i = \{1,1+s,1+2s,\ldots,1+(k-2)s,1+i\}$.
\end{itemize}
Now, for each $s \leq i \leq n-s$, we know that $\sigma^i(1) = 1+i$ and therefore, $1+i \in \sigma^i(v_i)$ which implies that $\{v_i, \sigma^i(v_i)\}$ is not an edge of $\operatorname{KG}(n,k)_{s-\operatorname{stab}}$. Thus, for $s \leq i \leq n-s$, $\sigma^i$ is not a shift of $\operatorname{KG}(n,k)_{s-\operatorname{stab}}$.
\end{proof}

In the following lemma we consider $[0]=\emptyset$. 

\begin{lemma}\label{lema-shift2}
Let $sk+1\leq n\leq s(k+1)-2$ and $r=n-sk$. Then, the shifts of the $s$-stable Kneser graph $\operatorname{KG}(n,k)_{s-\operatorname{stab}}$ are the rotations $\sigma^i$ for $i\in\{1,\dots,s-1\}\cup\{n-s+1,\dots,n-1\}\cup\left[\bigcup_{m\in [k-2]}\{ms+r+1,\dots,(m+1)s-1\}\right]$.
\end{lemma}
\begin{proof}
Let $T=\{1,\dots,s-1\}\cup\{n-s+1,\dots,n-1\}\cup\left[\bigcup_{m\in [k-2]}\{ms+r+1,\dots,\right.$ $\left. (m+1)s-1\}\right]$.
From Lemma \ref{reflexions} we know that the reflexions are not shifts. It is not hard to see that the circular permutations $\sigma^i$ with $i\in\{1,\dots,s-1\}\cup\{n-s+1,\dots,n-1\}$ are shifts of the graph $\operatorname{KG}(n,k)_{s-\operatorname{stab}}$. So, let $i\in\bigcup_{m\in [k-2]}\{ms+r+1,\dots,(m+1)s-1\}$. If $v,\sigma^i(v)$ are not adjacent for some vertex $v$, then there exist $j\in v\cap\sigma^i(v)$. Therefore, $\{j,j-i\}\subset v$. From the symmetry of $\operatorname{KG}(n,k)_{s-\operatorname{stab}}$, w.l.o.g. we assume that $\{1+i,1\}\subset v$. Notice that $|v\cap [i]|\leq\left\lfloor \frac{i}{s}\right\rfloor$ and $$|v\cap \{1+i,\dots,n\}|\leq\left\lfloor \frac{n-i}{s}\right\rfloor.$$

Consider $m'\in [k-2]$ such that $i\in\{m's+r+1,\dots,(m'+1)s-1\}$. Then,
\begin{itemize}
\item
$\left\lfloor \frac{i}{s}\right\rfloor\leq\left\lfloor \frac{(m'+1)s-1}{s}\right\rfloor =m'$.
\item
$\left\lfloor \frac{n-i}{s}\right\rfloor\leq\left\lfloor \frac{n-(m's+r+1)}{s}\right\rfloor \leq
\left\lfloor \frac{n-r-1}{s}\right\rfloor -m' = \left\lfloor \frac{n-n+sk-1}{s}\right\rfloor -m'=k-1-m'$.
\end{itemize}

Thus, $|v|\leq\left\lfloor \frac{i}{s}\right\rfloor + \left\lfloor \frac{n-i}{s}\right\rfloor\leq k-1$ which is a contradiction. Therefore $\sigma^i$ is a shift.

Now, let us see that if $i\notin T$, $\sigma^i$ is not a shift of $\operatorname{KG}(n,k)_{s-\operatorname{stab}}$.

Let $F_d=\{ds,ds+1,\dots,ds+r\}$ for $d\in [k-1]$ and $F=\bigcup_{d=1}^{k-1}F_d$. Observe that $F=[n]-T$.

Let $i\in F_d$ for some $d\in [k-1]$. Consider $t=i-ds$ and $v=\{1,1+s+t,1+2s+t,\dots,1+(k-1)s+t\}$. Then $v$ is a vertex of $\operatorname{KG}(n,k)_{s-\operatorname{stab}}$ and $\{v,\sigma^i(v)\}$ is not an edge of $\operatorname{KG}(n,k)_{s-\operatorname{stab}}$ since $\sigma^i(1)=1+i=1+ds+t$ belongs to $v$. Therefore, if $i\in F$ the rotations $\sigma^i$ is not a shift of $\operatorname{KG}(n,k)_{s-\operatorname{stab}}$ and the result follows.
\end{proof}

As a by-product of these results, in the following section we prove that if $n\geq 2k+2$, the Schrijver graphs $\operatorname{KG}(n,k)_{2-\operatorname{stab}}$ are not weakly hom-idempotent.

\section{Almost all Schrijver graphs are not weakly hom-idempotent}\label{Schrijver}

Given a graph $G$, the set of all shifts of $G$ is denoted by $S_G$.
Larose et al. \cite{LLT98} showed the following useful results:

\begin{proposition}[Proposition 2.3 in \cite{LLT98}]
\label{prop-core}
A graph $G$ is hom-idempotent if and only if $G \leftrightarrow \operatorname{Cay}(\operatorname{Aut}(G^{\bullet}),S_{G^{\bullet}})$.
\end{proposition}

\begin{theorem}[Theorem 5.1 in \cite{LLT98}]
\label{teo-critical}
Let $G$ be a $\chi$-critical graph. Then $G$ is weakly hom-idempotent if and only if it is hom-idempotent. 
\end{theorem}

\begin{proposition}
\label{prop-no-hom}
Let $n \geq 2k+2$ and let $G$ denote the graph $\operatorname{KG}(n,k)_{2-\operatorname{stab}}$. Then, $G \not\to \operatorname{Cay}(\operatorname{Aut}(G), S_G)$.
\end{proposition}
\begin{proof}
We know that the automorphism group of the graph $\operatorname{KG}(n,k)_{2-\operatorname{stab}}$ is the dihedral group $D_{2n}$ on $[n]$. Moreover, by Lemma \ref{lema-shift}, we known that the only two shifts of $\operatorname{KG}(n,k)_{2-\operatorname{stab}}$ are the circular permutations $\sigma$ and $\sigma^{-1}$.
Therefore the Cayley graph $\operatorname{Cay}(D_{2n}, \{\sigma, \sigma^{-1}\})$ is a disjoint union of two n-cycles. This implies that $2 \leq \chi(\operatorname{Cay}(D_{2n}, \{\sigma, \sigma^{-1}\})) \leq 3$.
Thus $\operatorname{KG}(n,k)_{2-\operatorname{stab}} \not\to \operatorname{Cay}(D_{2n}, \{\sigma, \sigma^{-1}\})$.
\end{proof}

As mentioned in the previous section, we know that any $2$-stable Kneser graph is a core. Therefore, by Propositions \ref{prop-core} and \ref{prop-no-hom}, and by Theorem \ref{teo-critical}, we have the following result.

\begin{theorem}
\label{teo-no-weak}
For any $n \geq 2k+2$, the $2$-stable Kneser graphs $\operatorname{KG}(n,k)_{2-\operatorname{stab}}$ are not weakly hom-idempotent.
\end{theorem}
\section{$s$-stable Kneser graphs $\operatorname{KG}(ks+1,k)_{s-\operatorname{stab}}$}
Let $\overline{G}$ denote the complement graph of the graph $G$, i.e. $\overline{G}$ has the same vertex set of $G$ and two vertices are adjacent in $\overline{G}$ if and only if they are not adjacent in $G$. Let $p$ be a positive integer. The $p$th power of a graph $G$, that we denoted by $G^{(p)}$, is the graph having the same vertex set as $G$ and where two vertices are adjacent in $G^{(p)}$ if the distance between them in $G$ is at most equal to $p$, where the distance of two vertices in a graph $G$ is the number of edges on the shortest path connecting them.

Let $n \geq 2k$ be positive integers. The Cayley graphs $\operatorname{Cay}(\mathbb{Z}_n,\{k,k+1,\ldots,n-k\})$, that we denoted by $G(n,k)$, are known as {\em circular graphs} \cite{Vince88,HT97}, where $\mathbb{Z}_n$ denote the cyclic group of order $n$. It is well known that the Kneser graph $\operatorname{KG}(n,k)$ contains an induced subgraph isomorphic to $G(n,k)$. In fact, let $C(n,k)$ be the subgraph of $\operatorname{KG}(n,k)$ obtained by restricting the vertex set of $\operatorname{KG}(n,k)$ to the shifts modulo $[n]$ of the $k$-subset $\{1,2,\ldots,k\}$, that is, $\{1,2,\ldots,k\}, \{2,3,\ldots,k+1\},\ldots,\{n,1,2,\ldots,k-1\}$. Define $\phi : G(n,k) \to C(n,k)$ by putting $\phi(u) = \{u+1,u+2,\ldots,u+k\}$ where the arithmetic operations are taken modulo $[n]$. Clearly, $\phi$ is a graph isomorphism. Notice also that the graph $G(n,k)$ is isomorphic to the graph $\overline{C_n^{(k-1)}}$, i.e. the complement graph of the $(k-1)$th power of a cycle $C_n$. Vince \cite{Vince88} has shown that $\chi(G(n,k)) = \lceil \frac
 {n}{k} \rceil$. \\

In the remainder of this section, we will always assume w.l.o.g. that any vertex $v=\{v_1,v_2,\ldots,v_k\}$ of the $s$-stable Kneser graph $\operatorname{KG}(ks+1,k)_{s-\operatorname{stab}}$ is such that $v_1 < v_2 < \ldots < v_k$, where $s,k \geq 2$. For $i\in [k-1]$, let $l_i(v)=v_{i+1}-v_i$ and $l_k(v)=v_1+(ks+1)-v_k$.
If $C$ is the cycle on $ks+1$ points labeled by integers $1,2,\ldots,ks+1$ in the clockwise direction and $v=\{v_1,v_2,\ldots,v_k\}$ is a vertex of the $s$-stable Kneser graph $\operatorname{KG}(ks+1,k)_{s-\operatorname{stab}}$, then $l_i(v)$ gives the distance in the clockwise direction between $v_i$ and $v_{i+1}$ in $C$.

\begin{lemma}
\label{lema-pareja}
Let $s,k \geq 2$ and let $v=\{v_1,v_2,\ldots,v_k\}$ be a vertex of $\operatorname{KG}(ks+1,k)_{s-\operatorname{stab}}$. Then, $l_i(v)\in\{s,s+1\}$ for all $i\in [k]$. Moreover, there exists exactly one $i'\in [k]$ such that $l_{i'}(v)=s+1$.
\end{lemma}
\begin{proof}
By definition, $l_i(v)\geq s$ for any $i\in [k]$. The result follows from the fact that $\sum_{i=1}^kl_i(v)=ks+1$.
\end{proof}

\begin{lemma}
\label{lema-numero-vertices}
Let $s,k \geq 2$. The number of vertices of the graph $\operatorname{KG}(ks+1,k)_{s-\operatorname{stab}}$ is equal to $ks+1$.
\end{lemma}
\begin{proof}
Again, let $C$ be the cycle on $ks+1$ points labeled by integers $1,2,\ldots,ks+1$ in the clockwise direction. From Lemma \ref{lema-pareja}, we have that each vertex of $\operatorname{KG}(ks+1,k)_{s-\operatorname{stab}}$ is uniquely determined by a clockwise circular interval of length $s+1$ in $C$. Trivially there exist $ks+1$ distinct clockwise circular intervals of length $s+1$ in $C$ and the lemma holds.
\end{proof}

\begin{proposition}
\label{prop-iso}
Let $s,k \geq 2$. Then, $G(ks+1,k) \simeq \operatorname{KG}(ks+1,k)_{s-\operatorname{stab}}$.
\end{proposition}
\begin{proof}
Let $C$ be a cycle on $ks+1$ points. We assume that the vertices of $G(ks+1,k)$ are disposed over $C$ in clockwise increasing order from $0$ to $ks$. In order to prove the isomorphism, we define the application $\phi : G(ks+1,k) \to \operatorname{KG}(ks+1,k)_{s-\operatorname{stab}}$ as follows: let $u$ be a vertex of $G(ks+1,k)$ such that $u = jk+i$, where $0 \leq j \leq s-1$ and $0 \leq i \leq k-1$. Then, $\phi(u) = \{u_1,\ldots,u_k\}$ where,
$$u_r=
\left\{
\begin{array}{ll}
j+1+(r-1)s,&\operatorname{if } 1 \leq r \leq  k-i\\
j+2+(r-1)s,&\operatorname{if } k-i+1 \leq r \leq  k.
\end{array}
\right.$$

Finally, define $\phi(ks) = \{s+1,2s+1,\ldots,ks+1\}$. From Lemma \ref{lema-numero-vertices}, it is not difficult to prove that $\phi$ is a bijective function. It remains to show that $\phi$ is indeed a graph isomorphism. Let $u,v$ be two vertices in $C(ks+1,k)$. In the sequel, we assume that $v > u$. In fact, if $u >v$ we can always swap $u$ and $v$. Let $u = jk+i$, where $0 \leq j \leq s-1$ and $0 \leq i \leq k-1$. Let $t = v - u$, where $1 \leq t \leq sk$. Let us see that $\phi(u), \phi(v)$ in $\operatorname{KG}(ks+1,k)_{s-\operatorname{stab}}$ are adjacent if and only if 
$k\leq t\leq k(s-1)+1$. Consider $t=xk+y$ where $0\leq x\leq s$ and $0\leq y\leq k-1$. Besides, let $V_y=\{k-i+1-y,\dots,k-i\}$ if $1\leq y\leq k-i$ and $V_y=\{1,\dots,k-i,k(s+1)-i+1-y,\dots,ks\}$ if $y>k-i$. By construction, notice that:
\begin{itemize}
\item if $1\leq y\leq k-1$ then $v_r=u_r+1+x$ if $r\in V_y$ and $v_r=u_r+x$ if $r\notin V_y$.
\item if $y=0$ then $v_r=u_r+x$ for all $r$.
\end{itemize}

Therefore, if $t\leq k-1$ then $v_r=u_r$ for all $r\notin V_y$. So, we have that $\phi(u) \cap \phi(v) \neq \emptyset$. Analogously, if $k(s-1)+2 \leq t \leq ks$ then $v_r=u_{r+1}$ for all $r\in V_y\setminus\{k-i\}$. Again, we have that $\phi(u) \cap \phi(v) \neq \emptyset$. Besides, notice that $l_r(u)=s$ if $r\neq k-i$ and $l_{k-i}(u)=s+1$. From this fact, it follows that $\phi(u) \cap \phi(v)=\emptyset$ if $k\leq t\leq k(s-1)+1$.

Therefore, vertices $u,v$ in $C(ks+1,k)$ are adjacent if and only if vertices $\phi(u), \phi(v)$ in $\operatorname{KG}(ks+1,k)_{s-\operatorname{stab}}$ are adjacent. 
\end{proof}

A direct consequence of Proposition \ref{prop-iso} is that $\chi(\operatorname{KG}(ks+1,k)_{s-\operatorname{stab}}) = s+1$. In fact, Vince \cite{Vince88} has shown, at the end of the eighties, that $\chi(G(n,k)) = \lceil \frac{n}{k} \rceil$, and thus, we obtain that $\chi(G(ks+1,k)) = \chi(\operatorname{KG}(ks+1,k)_{s-\operatorname{stab}}) = s+1$. However, as far as we know, there was no known connections between graphs $G(ks+1,k)$ and $\operatorname{KG}(ks+1,k)_{s-\operatorname{stab}}$. For this reason, twenty years later, Meunier (see Proposition 1 in \cite{Meun11}) computes again the chromatic number of $\operatorname{KG}(ks+1,k)_{s-\operatorname{stab}}$.\\

Let $\operatorname{Cay}(A,S)$ be a Cayley graph. If $a^{-1}Sa = S$ for all $a \in A$, then $Cay(A,S)$ is called a {\em normal Cayley graph}.

\begin{lemma}[\cite{HHP95}]
\label{lema-cayley}
Any normal Cayley graph is hom-idempotent.
\end{lemma}

Note that all Cayley graphs on abelian groups are normal, and thus hom-idempotents. In particular, the {\em circulant} graphs are Cayley graphs on cyclic groups (i.e., cycles, powers of cycles, complements of powers of cycles, complete graphs, etc). Therefore, by Proposition \ref{prop-iso} and Lemma \ref{lema-cayley}, we have the following result.

\begin{theorem}
\label{teo-hom-idemp}
Let $s,k \geq 2$. Then, $\operatorname{KG}(ks+1,k)_{s-\operatorname{stab}}$ is hom-idempotent.
\end{theorem}

\section{Properties of the graph $\operatorname{KG}(2s+2,2)_{s-\operatorname{stab}}$ }

In this section, we study some properties of the graph $\operatorname{KG}(2s+2,2)_{s-\operatorname{stab}}$ , with $s \geq 3$. First recall that in Section \ref{Schrijver}, we use the strong structural property of "criticality" of Schrijver graphs to prove that almost all Schrijver graphs are not weakly hom-idempotent. We will prove in this section  that graphs $\operatorname{KG}(2s+2,2)_{s-\operatorname{stab}}$ are not $\chi$-critical for all $s \geq 3$. However, we will prove that these graphs are core and thus, we will be able to deduce that $\operatorname{KG}(2s+2,2)_{s-\operatorname{stab}}$ is not hom-idempotent for all $s \geq 3$. \\

In 2011, Meunier \cite{Meun11} has settled a conjecture concerning the chromatic number of $r$-uniform $s$-stable Kneser hypergraphs, with $r,s \geq 2$,  which is still an open problem, even for $2$-uniform $s$-stable Kneser (hyper)graphs. For $2$-uniform $s$-stable Kneser (hyper)graphs, the conjecture can be expressed as follows:

\begin{conjecture}[\cite{Meun11}]
\label{chiKG}
$\chi(\operatorname{KG}(n,k)_{s-\operatorname{stab}})=n-(k-1)s$, for any $s,k \geq2$ and $n > sk$.
\end{conjecture}

Since this conjecture was stated, some papers have confirmed it for particular cases (see, e.g \cite{Jonsson12,Meun11}). However, the case $k=2$ and $n=2s+2$ is still open. We will prove that Meunier's Conjecture \ref{chiKG} holds for the case $k = 2$, $n = 2s+2$, and any $s \geq 3$. In fact, we will show that $\chi(\operatorname{KG}(2s+2,2)_{s-\operatorname{stab}})=s+2$ .

\medskip

Let us consider the $s$-stable Kneser graphs $\operatorname{KG}(2s+2,2)_{s-\operatorname{stab}}$ for $s\geq3$.
It is known that $s+1\leq\chi(\operatorname{KG}(2s+2,2)_{s-\operatorname{stab}})\leq s+2$. 
Let $\left\{I_i\right\}_{i\in [2s+2]}$ be the family of all maximum stable sets of $\operatorname{KG}(2s+2,2)_{s-\operatorname{stab}}$, where $I_i$ has center $i$, for $i\in [2s+2]$ (see Theorem 3 in \cite{Talbot03}).

\medskip

Let $S=\{\{1,1+s\},\{2,2+s\},\ldots,\{s+2,s+2+s\},\{1,3+s\},\{2,4+s\},\ldots,\{s,2s+2\}\}$ and $G_n$ be the Cayley graphs $\operatorname{Cay}(\mathbb{Z}_n,\{\pm 1,\pm 2,\ldots,\pm (s-1),s+1\})$ with $n=2s+2$. Let us denote $KG[S]$ the subgraph induced by $S$ in $\operatorname{KG}(2s+2,2)_{s-\operatorname{stab}}$. Observe that $KG[S]$ is isomorphic to $G_n$, with the isomorphism $\phi:G_n\mapsto KG[S]$ defined as follows:

\medskip

$\phi(u)=\left\{
\begin{array}[h]{ll}
	\{u+1,u+1+s\} & \text{if } u\in\{0,\ldots,s+1\};\\
	\{u-(s+1),u+1\} & \text{if } u\in\{s+2,\ldots,2s+1\}.
\end{array}
\right.$

\medskip

Besides, notice that $C_n^{(s-1)}$ is a subgraph (not induced) of $KG[S]$, therefore $\chi(KG[S])\geq \chi(C_n^{(s-1)})=s+1$. This last fact holds from the following known result.

\begin{theorem}[\cite{ProwWood03}]
\label{chiC}
Let $n\geq2a$ and $n=q(a+1)+r$, with $q>0$ and $0\leq r\leq a$. Then, $\chi(C_n^a)=a+1+\left\lceil \frac{r}{q}\right\rceil$.
\end{theorem}

On the other hand, the set $T=\{\{1,2+s\},\{2,3+s\},\ldots,\{s+1,2s+2\}\}$ induces a complete graph in $\operatorname{KG}(2s+2,2)_{s-\operatorname{stab}}$ and $S,T$ induce a partition of $V(\operatorname{KG}(2s+2,2)_{s-\operatorname{stab}})$.

\medskip

Let $G$ be the subgraph induced by $S\cup\{\{1,2+s\},\{2,3+s\}\}$. We will prove that $\chi(G)\geq s+2$. Assume that $\chi(G)=s+1$. Let $f$ be a minimum coloring of $G$. Since $\alpha(KG[S])=2$, each color class of $f$ has exactly two vertices in $KG[S]$. Besides, $f^{-1}(f(\{1,2+s\}))$ and $f^{-1}(f(\{2,3+s\}))$ are disjoint maximum stable sets in $G$. Then, $f^{-1}(f(\{1,2+s\}))=I_1$ and $f^{-1}(f(\{2,3+s\}))=I_2$ or $f^{-1}(f(\{1,2+s\}))=I_{2+s}$ and $f^{-1}(f(\{2,3+s\}))=I_{3+s}$. W.l.o.g. we assume that $f^{-1}(f(\{1,2+s\}))=I_1$ and $f^{-1}(f(\{2,3+s\}))=I_2$. Therefore, $f(\{1,1+s\})=f(\{1,3+s\})=f(\{1,2+s\})$ and $f(\{2,2+s\})=f(\{2,4+s\})=f(\{2,3+s\})$. Let $a=f(\{1,2+s\})$ and $b=f(\{2,3+s\})$. Let $N(v)$ be the set of neighbors of vertex $v$ in $G$. The set $U=\{\{3,3+s\},\{4,4+s\},\ldots,\{s+2,s+2+s\}\}$ verifies that $U\subset N(\{1,1+s\})\cup N(\{1,3+s\})$ and $U\subset N(\{2,2+s\})\cup N(\{2,4+s\})$. Then $f(v)\notin\{a,b\}$ for all $v\in U$. Since $U$ has cardinality $s$ and induces a complete graph in $G$, $f$ need at least $s+2$ colors, which is a contradiction.

\medskip

Thus, we obtain the following lemma.
\begin{lemma}
For all $s\geq3$, $\chi(\operatorname{KG}(2s+2,2)_{s-\operatorname{stab}})=s+2$ and $\operatorname{KG}(2s+2,2)_{s-\operatorname{stab}}$ is not $\chi$-critical.
\end{lemma}

\medskip

However, let us see that $\operatorname{KG}(2s+2,2)_{s-\operatorname{stab}}$ is a core. Firstly, notice that the chromatic number of $\operatorname{KG}(2s+2,2)_{s-\operatorname{stab}}-\{s+2,2s+2\}$ is $s+1$, since its vertex set admits the partition $I_1,\ldots,I_{s},J$ with $J=\{\{s+1,2s+1\},\{s+1,2s+2\}\}$. Besides, since $\operatorname{Aut}(\operatorname{KG}(2s+2,2)_{s-\operatorname{stab}})$ acts transitively on $S$, $\chi(\operatorname{KG}(2s+2,2)_{s-\operatorname{stab}}-v)=s+1$ for all $v\in S$.

\medskip

Therefore, $S$ is contained in the vertex set of the core of $\operatorname{KG}(2s+2,2)_{s-\operatorname{stab}}$. Assume that $\operatorname{KG}(2s+2,2)_{s-\operatorname{stab}}$ is not a core and let $G'$ be its core. Then, there is a retraction $f$ of $\operatorname{KG}(2s+2,2)_{s-\operatorname{stab}}$ onto $G'$ \cite{HN92}. It follows that $f(u)\in S$ for some vertex $u\in T$, since $T$ induces a complete graph in $\operatorname{KG}(2s+2,2)_{s-\operatorname{stab}}$. Notice that $u=\{i,i+s+1\}$ for some $i\in \{1,\ldots,s+1\}$. Then $f(u)\in\{\{i,i+s\},\{i,i+s+2\},\{i-1,i+s+1\},\{i+1,i+s+1\}\}$.

\medskip

Let us prove that if $f(u)\in\{\{i,i+s\},\{i,i+s+2\},\{i-1,i+s+1\},\{i+1,i+s+1\}\}$, there is a vertex $v\in S$ such that $u$ and $v$ are adjacent in $\operatorname{KG}(2s+2,2)_{s-\operatorname{stab}}$ but $f(u)$ and $f(v)$ are not adjacent in $G'$, which is a contradiction. Observe that $f(v)=v$ since $f$ is a retraction of $\operatorname{KG}(2s+2,2)_{s-\operatorname{stab}}$ onto $G'$. Then, for $f(u)=\{i,i+s\},\{i,i+s+2\},\{i-1,i+s+1\},\{i+1,i+s+1\}$ let $v=\{i-2,i+s\},\{i+2,i+s+2\},\{i-1,i+s-1\},\{i+1,i+s+3\}$, respectively.

\medskip

Therefore, we obtain the following result.
\begin{lemma}
$\operatorname{KG}(2s+2,2)_{s-\operatorname{stab}}$ is a core.
\end{lemma}

\medskip


In order to obtain that $\operatorname{KG}(2s+2,2)_{s-\operatorname{stab}}$ are not hom-idempotent, we can follow the reasoning in Section \ref{Schrijver}, since $s$-stable Kneser graphs $\operatorname{KG}(2s+2,2)_{s-\operatorname{stab}}$ are cores. Firstly, we will observe that if $G=\operatorname{KG}(2s+2,2)_{s-\operatorname{stab}}$, the graphs $\mbox{Cay}(\mbox{Aut}(G),S_{G})$ is isomorphic to the disjoint union of two $C_n^{s-1}$. Moreover, we prove a more general result.

\begin{remark}
\label{iso}
If $G=\operatorname{KG}(2s+2,2)_{s-\operatorname{stab}}$ or $G=\operatorname{KG}(n,k)_{s-\operatorname{stab}}$ with $n\geq (k+1)s-1$, $\operatorname{Cay}(\operatorname{Aut}(G),S_{G})$ is isomorphic to the disjoint union of two $C_n^{s-1}$.
\end{remark}
\begin{proof}
It is easy to see that $\sigma^i\sigma^j=\sigma^{i+j}$. Besides, if $n$ is odd we have:

\bigskip

$\sigma^j\rho_i=\rho_m$ with

\bigskip

$m=\left\{
\begin{array}[h]{ll}
	i+\frac{n-1}{2}+\frac{j+1}{2} & \text{if}\ j\ \text{is odd},\\
	i+\frac{j}{2} & \text{if}\ j\ \text{is even},
\end{array}\right.$

\bigskip

and if $n$ is even:

\bigskip

$\sigma^j\rho_i=\left\{
\begin{array}[h]{ll}
	\delta_m & \text{if}\ j\ \text{is odd, with}\ m=i+\frac{j+1}{2}\ (\mbox{mod}\ \frac{n}{2}),\\
	\rho_m & \text{if}\ j\ \text{is even, with}\ m=i+\frac{j}{2}\ (\mbox{mod}\ \frac{n}{2}).
\end{array}\right.$

\bigskip

Therefore, since $\operatorname{Aut}(G)$ is isomorphic to $D_{2n}$ and $S_{G}=\{\sigma^i:i=1,\dots,s-1,n-s+1,\dots,n-1\}$, from previous facts it follows that the rotations induce a $C_n^{s-1}$ and the reflexions also induce a $C_n^{s-1}$.
\end{proof}

Finaly, to prove that if $G=\operatorname{KG}(2s+2,2)_{s-\operatorname{stab}}$ then $G \not\to \operatorname{Cay}(\operatorname{Aut}(G), S_G)$, it is enough to notice that, from Theorem \ref{chiC}, $\chi(C_{2s+2}^{s-1})=s+1<s+2$. So, by Proposition \ref{prop-core}, we obtain the following result.

\begin{lemma}
For $s \geq 3$, $\operatorname{KG}(2s+2,2)_{s-\operatorname{stab}}$ is not hom-idempotent.
\end{lemma}

\medskip

Observe that if Conjecture \ref{chiKG} is true for $n\geq (k+1)s-1$ and the graphs $\operatorname{KG}(n,k)_{s-\operatorname{stab}}$ are cores, by an analogous reasoning as before we obtain that the following conjecture is true.

\begin{conjecture}
If $n\geq (k+1)s-1$ and $s\geq3$, the $s$-stable Kneser gragh $\operatorname{KG}(n,k)_{s-\operatorname{stab}}$ is not hom-idempotent.
\end{conjecture}

Finally, we end this paper with a more strong conjecture:

\begin{conjecture}
Let $s \geq 3$, $k \geq 2$ and $n > ks + 1$. Then, the $s$-stable Kneser gragh $\operatorname{KG}(n,k)_{s-\operatorname{stab}}$ is not hom-idempotent.
\end{conjecture}




\bibliographystyle{plain}



\end{document}